\documentclass[12pt]{amsart}

\usepackage{amscd,latexsym,amsthm,amsfonts,amssymb,amsmath,amsxtra}
\usepackage[mathscr]{eucal}
\usepackage{tkz-euclide}
\usepackage{pst-plot}
\usepackage{hyperref}
\renewcommand\theequation{\thesection.\arabic{equation}}

\newcommand{\BC}{{\mathbb {C}}}

\newcommand{\BZ}{{\mathbb {Z}}}

\newcommand{\CH}{{\mathcal {H}}}

\newcommand{\CJ}{{\mathcal {J}}}

\newcommand{\CW}{{\mathcal {W}}}

\newcommand{\GL}{{\mathrm{GL}}}

\newcommand{\Hom}{{\mathrm{Hom}}}

\newcommand{\Ind}{{\mathrm{Ind}}}

\newcommand{\SO}{{\mathrm{SO}}}

\newcommand{\wt}{\widetilde}

\newcommand{\ol}{\overline}

\newcommand{\bs}{\backslash}

\def\gam{{\gamma}}

\newtheorem{thm}{Theorem}[section]

\newtheorem{lem}[thm]{Lemma}
\newtheorem{prop}[thm]{Proposition}
\newtheorem {conj}[thm]{Conjecture}

\newtheorem {ques/conj}[thm]{Question/Conjecture}

\newtheorem{defn}[thm]{Definition}

\newtheorem*{Jconj}{Conjecture~$\bJ(\textbf{\textit{n}},\br)$}

\makeatletter

\newcommand{\Rmnum}[1]{\expandafter\@slowromancap\romannumeral #1@}
\makeatother

\def\bJ{{\boldsymbol{\CJ}}}
\def\bN{{\boldsymbol{N}}}
\def\br{{\boldsymbol{r}}}

\begin{document}
\renewcommand{\theequation}{\arabic{equation}}
\numberwithin{equation}{section}

\title[Local Converse Theorem]{On the Local Converse Theorem for $p$-adic $\GL_n$}

\author{Herv\'e Jacquet}
\address{Department of Mathematics\\
Columbia University\\
Rm 615, MC 4408 2990 Broadway\\
New York, NY 10027 USA}
\email{hj@math.columbia.edu}

\author{Baiying Liu}
\address{Department of Mathematics\\
Purdue University\\
150 N. University Street\\
West Lafayette, IN 47907 USA}
\email{liu2053@purdue.edu}

\begin{abstract}
In this paper, we completely prove a standard conjecture on the local converse theorem for generic representations of $\GL_n(F)$, where $F$ is a non-archimedean local field.
\end{abstract}

\date{\today}
\thanks{The second mentioned author is supported in part by NSF Grants DMS-1302122, DMS-1620329.}
\maketitle

\section{Introduction}\label{intro}

Let $F$ be a non-archimedean local field.
Let $G_n:=\GL_n(F)$ and let $\pi$ be an
irreducible generic representation of $G_n$. The family
of local gamma factors $\gamma(s, \pi \times \tau, \psi)$, for $\tau$ any irreducible generic representation of $G_r$, $\psi$ an additive character of $F$ and $s\in\BC$, can be
defined using Rankin--Selberg convolution \cite{JPSS83} or the
Langlands--Shahidi method \cite{S84}. The following is a standard conjecture on precisely which family of gamma factors determine $\pi$. 

\begin{conj}[]\label{lcp}
Let $\pi_1,\pi_2$ be irreducible generic representations
of $G_n$. Suppose that they have the same central character.
If
\[
\gamma(s, \pi_1 \times \tau, \psi) = \gamma(s, \pi_2 \times \tau, \psi),
\]
as functions of the complex variable $s$, for all irreducible
generic representations $\tau$ of $G_r$ with $1 \leq r \leq 
[\frac{n}{2}]$, then $\pi_1 \cong \pi_2$.
\end{conj}

The conjecture and the global version of the conjecture (\cite[Section 8, Conjecture 1]{CPS99}) emerged from early discussions between Piatetski-Shapiro, Shalika and the first mentioned author. In particular, they proved Conjecture \ref{lcp} in the case $n=3$ (\cite{JPSS79}).

The fact that the representations have the same central character implies that if, for a given $r$, the above equality is true for one choice of  $\psi$, then it is true for all choices of $\psi$. Moreover, if the above equality is true for $r = 1$ and one choice of $\psi$, then the representations have the same central character (\cite[Corollary 2.7]{JNS15}).

One can propose a more general family of conjectures as follows (see \cite{ALSX16}).
We say that $\pi_1$ and $\pi_2$ satisfy hypothesis $\CH_0$
if they have the same central character. 
For $m \in \BZ_{\geq 1}$, we say that they satisfy hypothesis $\CH_m$ if they satisfy hypothesis $\CH_0$ and satisfy
$$\gam(s,\pi_1\times\tau,\psi)=\gam(s,\pi_2\times\tau,\psi)$$
as functions
of the complex variable $s$, for all irreducible generic
representations $\tau$ of $G_m$.
For $r \in \BZ_{\geq 0}$, we say that $\pi_1,\pi_2$ satisfy hypothesis $\CH_{\le r}$ if they satisfy hypothesis $\CH_m$,
for $0\le m\le r$.

\begin{Jconj}
If $\pi_1,\pi_2$ are irreducible generic representations
of $G_n$ which satisfy hypothesis $\CH_{\le r}$, then $\pi_1\simeq\pi_2$.
\end{Jconj}

Conjecture \ref{lcp} is exactly Conjecture $\CJ(n,[\frac{n}{2}])$. Henniart proved Conjecture $\CJ(n,n-1)$ in \cite{H93}.  Conjecture $\CJ(n,n-2)$ (for $n \geq 3$) is a theorem due
to Chen \cite{Ch96, Ch06}, to Cogdell and Piatetski-Shapiro \cite{CPS99}, and to Hakim and Offen \cite{HO15}. As we mentioned above, Conjecture $\CJ(3,1)$ is first proved by the first mentioned author, Piatetski-Shapiro, and Shalika \cite{JPSS79}. Conjecture $\CJ(2,1)$ is first proved by the first mentioned author and Langlands \cite{JL70}.

In \cite[Section 2.4]{JNS15}, Conjecture \ref{lcp} is shown to be equivalent to the same conjecture with the adjective ``generic" replaced by ``unitarizable supercuspidal" as follows: 

\begin{conj}\label{lcp2}
Let $\pi_1,\pi_2$ be irreducible unitarizable supercuspidal representations of $G_n$. Suppose that they have the same central character. If
\[
\gamma(s, \pi_1 \times \tau, \psi) = \gamma(s, \pi_2 \times \tau, \psi),
\]
as functions of the complex variable $s$, for all irreducible
supercuspidal representations $\tau$ of $G_r$ with $1 \leq r \leq 
[\frac{n}{2}]$, then $\pi_1 \cong \pi_2$.
\end{conj}

In \cite{JNS15}, Jiang, Nien and Stevens introduced
the notion of a special pair of Whittaker functions for a pair of irreducible unitarizable supercuspidal representations $\pi_1$, $\pi_2$ of $G_n$. They proved that if there is such a pair, and $\pi_1$, $\pi_2$ satisfy hypothesis $\CH_{\leq [\frac{n}{2}]}$, then $\pi_1 \cong \pi_2$. They also found special pairs of Whittaker functions in many cases, in particular the case of depth zero representations. In \cite{ALSX16}, Adrian, the second mentioned author, Stevens and Xu proved part of the case left open in \cite{JNS15}.
In particular, the results in \cite{JNS15} and \cite{ALSX16} together imply that Conjecture \ref{lcp2} is true for $G_n$, $n$ prime. 
We remark that both \cite{JNS15} and \cite{ALSX16} make use of the construction of supercuspidal representations of $G_n$ in \cite{BK93} and properties of Whittaker functions of supercuspidal representations constructed in \cite{PS08}.  

In this paper we prove Conjecture 1.1, hence Conjecture 1.2. We use analytic methods. We do not resort to the construction of special pairs of Whittaker functions for supercuspidal representations.
The idea is inspired by the proof of Conjecture $\CJ(n,n-2)$ in \cite{Ch06}. We state the main result of the paper as the following theorem. 

\begin{thm}\label{main}
Conjecture \ref{lcp} is true. 
\end{thm}

We were recently informed that Chai has an independent and different proof of Conjecture \ref{lcp} (\cite{Ch16}). 

One straightforward application of Theorem \ref{main} is that it reduces the amount of necessary $\GL$-twisted local factors, in order to obtain the uniqueness of local Langlands correspondence (proved by Henniart in \cite{H02}), and it also gives a corresponding local converse theorem for local Langlands parameters via the local Langlands correspondence. 

By the argument of \cite[Section 7, Theorem]{CPS99}, one can see that Conjecture \ref{lcp2} is a consequence of the global version of Conjecture \ref{lcp} (\cite[Section 8, Conjecture 1]{CPS99}). Hence, Theorem \ref{main} provides evidence for the global version of Conjecture \ref{lcp} on the global converse theorem. 

It is easy to find pairs of generic representations showing that in Conjecture 1.1, $[\frac{n}{2}]$ is sharp for the generic dual of $G_n$. In \cite{ALST16}, we showed that, in Conjecture \ref{lcp2}, $[\frac{n}{2}]$ is sharp for the supercuspidal dual of $G_n$, for $n$ prime, in the tame case. It is believed that in Conjecture \ref{lcp2}, $[\frac{n}{2}]$ is sharp for the supercuspidal dual of $G_n$, for any $n$, in all cases.
This is our work in progress. 
However, it is expected that for certain families of supercuspidal representations, $[\frac{n}{2}]$ may not be sharp, for example, for simple supercuspidal representations (of depth $\frac{1}{n}$), the upper bound may be lowered to 1 (see \cite[Proposition 2.2]{BH14} and \cite[Remark 3.18]{AL16} in general, and \cite{X13} in the tame case).

Nien in \cite{N14} proved the finite fields analogue of Conjecture \ref{lcp}, using special properties of normalized Bessel functions. We remark that the idea in this paper also applies to the finite field case, and could give a new proof for the result in \cite{N14}. Moss in \cite{M16} proved an analogue of Conjecture $\CJ(n,n-1)$ for $\ell$-adic families of smooth representations of $\GL_n(F)$, where F is a finite extension of $Q_p$ and $\ell$ is different from $p$. 

The local converse problem has been studied for irreducible generic representations of groups other than $\GL_n$: $\mathrm{U}(2,1)$ and $\mathrm{GSp}(4)$ (Baruch,  \cite{B95} and \cite{B97}); $\SO(2n+1)$ (Jiang and Soudry,  \cite{JS03}); $\rm{U}(1,1)$ and $\rm{U}(2,2)$ (Zhang, \cite{Z15a} and \cite{Z15b}). 
We remark that since the local converse theorem for $\SO_{2n+1}$ in \cite{JS03} is eventually reduced to the local converse theorem for $\GL_{2n}$, following exactly the same proof given in \cite{JS03}, Theorem \ref{main} implies that twisting up to irreducible generic representations of $\GL_n$ is enough in the local converse theorem for $\SO_{2n+1}$ in \cite{JS03}.  

Section 2 will be preparation on properties of irreducible generic representations of $\GL_n(F)$ and Rankin-Selberg convolution. Theorem \ref{main} will be proved in Section 3.
Section 4 will be the proof of Proposition \ref{prop3}. 

Finally, we would like to thank J. Cogdell, D. Jiang and F. Shahidi for their interest in the problems discussed in this paper and for their encouragements, and S. Stevens for a helpful suggestion which makes the paper more readable.  We also would like to thank the referee for helpful comments and suggestions.

\subsection*{Acknowledgements} This material is based upon work supported by the National Science Foundation under agreement No. DMS-1128155. Any opinions, findings and conclusions or recommendations expressed in this material are those of the authors and do not necessarily reflect the views of the National Science Foundation.

\section{Generic representations and Rankin-Selberg convolution}

In this section, we review basic results on generic representations and the Rankin-Selberg convolution, which will be used in the proof of Theorem \ref{main} in Section 3. 

Let $F$ be a non-archimedean local field, and let $q$ be the cardinality of the residue field of $F$. 
Let $G_n:=\GL_n(F)$. All representations of $G_n$ considered in this paper are irreducible smooth and complex. 

\subsection{Whittaker models}

Let $B_n=T_nU_n$ be the standard Borel subgroup of $G_n$ consisting of upper triangular matrices, with unipotent radical $U:=U_n$ and diagonal group $T_n$. Fix a nontrivial additive character $\psi$ of $F$. Define a non-degenerate character $\psi_{U_n}$ of $U_n$ also denoted by $\psi_U$ as follows:
$$\psi_{U_n}(u):=\psi\left(\sum_{i=1}^{n-1} u_{i,i+1}\right)\,,\, u \in U_n\,.$$

An irreducible representation $(\pi,V)$ of $G_n$ is {\it generic} if 
$$\Hom_{G_n}(V, \Ind_{U}^{G_n} \psi_U) \neq 0\,.$$
It is known that if $\pi$ is generic, then the above Hom-space is of dimension 1. Let $\pi$ be an irreducible generic representation of $G_n$, fix a nonzero  functional $\ell$ in the above Hom-space, then the image of $V$ under $\ell$ is called the {\it Whittaker model} of $\pi$, denoted by $\CW(\pi,\psi)$. It is known that $\CW(\pi,\psi)$ is independent of the choice of $\ell$. For each $v \in V$, let $W_v=\ell(v)$. Then 
for $u \in U$, $g \in G_n$, 
$$W_v(ug)=\psi_U(u)W_v(g)\,,$$
$$W_v(g)=\ell(v)(g)= \ell(\pi(g)v)(I_n)=W_{\pi(g)v}(I_n)\,.$$
For $W \in \CW(\pi,\psi)$, let 
$$\wt{W}(g)=W(\omega_n {}^t g^{-1})\,,$$
where 
$$\omega_1=1\,,\, \omega_n=\begin{pmatrix}
0 & 1\\
\omega_{n-1} & 0
\end{pmatrix}\,.$$
It is well known that $\wt{W} \in \CW(\wt{\pi}, \ol{\psi})$, where $\wt{\pi}$ is the representation contragradient to $\pi$. 

Let $P$ be the maximal parabolic subgroup of $G_n$ with Levi subgroup $G_{n-1} \times G_1$. Let $Z$ be the center of $G_n$.  
Given two irreducible generic representations $\pi_1$ and $\pi_2$ of $G_n$ with the same central character, to show that $\pi_1 \cong \pi_2$, it suffices to show that their Whittaker models $\CW(\pi_1,\psi)$ and $\CW(\pi_2,\psi)$ have a nonzero intersection. The following two propositions allow us to study Whittaker functions by restricting them to $P$. 

\begin{prop}[\cite{GK75}]\label{prop1}
Let $\pi$ be an irreducible generic representation of $G_n$ with central character $\omega_{\pi}$. Then the restriction $\CW(\pi,\psi)|_P$ has a Jordan-H$\ddot{o}$lder series of finite length which contains the compact induction ${\rm ind}_{ZU}^P \omega_{\pi} \psi_U$ as an irreducible subrepresentation.
\end{prop}

The following proposition is proven in \cite{JPSS79} for $n=3$, and the same argument works for general $n$. The proof can also be found in \cite[Theorem 4.9]{BZ77}. 

\begin{prop}\label{prop2}
Let $(\pi, V)$ be an irreducible generic representation of $G_n$. Then
$$v \mapsto W_v|_P$$
is an injective map from $V$ to the space of smooth functions on $P$. 
\end{prop}

Let $\pi_1$, $\pi_2$ be two irreducible generic representations of $G_n$ with the same central character $\omega$. 
Let $V_0={ \rm ind}_{ZU}^P \omega\psi_U$. For $p\in P$, let $\rho(p)$ be the operator of right translation on complex functions $v$ on $P$:
$$\rho(p)v(x)=v(xp)\,.$$ 
By Propositions \ref{prop1} and \ref{prop2},
for any $v \in V_0$ there is a unique element $W^i_v$ in the Whittaker model of $\pi_i$ such that, for all $p \in P$, $W^i_v(p)=v(p)$. Thus, we have
$$W^1_v(p)=W^2_v(p)\,,\, \forall p \in P\,,\, \forall v \in V_0\,.$$
Note that for $p\in P$, we have
$$W^i_v(gp)=W^i_{\rho(p)v}(g)\,,\, \forall g \in G_n\,,\, \forall v \in V_0\,.$$

\subsection{Rankin-Selberg convolution}

Let $n, t \in \BZ_{\geq 1}$, and let $\pi$ and $\tau$ be irreducible generic representations of $G_n$ and $G_t$, with Whittaker models $\CW(\pi,\psi)$ and $\CW(\tau,\ol{\psi})$, respectively. Let $W \in \CW(\pi,\psi)$ and $W' \in \CW(\tau,\ol{\psi})$. Assume that $n > t$, which is the case of interest to us in this paper. 

For any integer $j$ with $0 \leq j \leq n-t-1$, let $k=n-t-1-j$, define a local zeta integral as follows:
\begin{equation}\label{zeta}
\Psi(s,W,W';j):=\int
\int
W \begin{pmatrix}
g & 0 & 0 \\
X & I_j & 0\\
0 & 0 & I_{k+1}
\end{pmatrix}W'(g) \lvert \det g \rvert^{s-\frac{n-t}{2}} dx dg\,,
\end{equation}
with integration being over $g \in U_t \bs G_t$ and $X \in M_{j \times t}(F)$. 

For $g \in G_n$, let $\rho(g)$ be the operator of right translation on complex functions $f$ on $G_n$:
$$\rho(g)f(x)=f(xg)\,.$$ 
Let 
$$\omega_{n,t}=\begin{pmatrix}
I_t & 0 \\
0 & \omega_{n-t}
\end{pmatrix}\,.$$

The following result is about functional equations for a pair of irreducible generic representations, proved by the first named author, Piatetski-Shapiro and Shalika in \cite{JPSS83}. It plays an important role in proving the main result of this paper.  

\begin{thm}[\cite{JPSS83}, Section 2.7]\label{rsconv}
With notation as above, the followings hold. 
\begin{enumerate}
\item Each integral $\Psi(s,W,W';j)$ is absolutely convergent for 
${\rm Re}(s)$ large and is a rational function of $q^{-s}$. More precisely, for any fixed $j$, the integrals $\Psi(s,W,W';j)$ span
a fractional ideal (independent of $j$) of $\BC[q^s, q^{-s}]$:
$$\BC[q^s, q^{-s}]L(s, \pi \times \tau)\,,$$
where the local factor $L(s, \pi \times \tau)$ has the form $P(q^{-s})^{-1}$, with $P \in \BC[X]$ and $P(0)=1$.

\item For any $0 \leq j \leq n-t-1$, there is a factor $\epsilon(s, \pi \times \tau, \psi)$, independent of $j$, such that 
$$\frac{\Psi(1-s, \rho(\omega_{n,t})\wt{W},\wt{W'};k)}{L(1-s, \wt{\pi} \times \wt{\tau})}
= \omega_{\tau}(-1)^{n-1} \epsilon(s, \pi \times \tau, \psi) \frac{\Psi(s,W,W';j)}{L(s, \pi \times \tau)}\,,$$
where $k=n-t-1-j$ and $\omega_{\tau}$ is the central character of $\tau$. 
\end{enumerate}
\end{thm}

The local gamma factor attached to a pair $(\pi, \tau)$ is defined to be 
$$\gamma(s, \pi \times \tau, \psi) = \epsilon(s, \pi \times \tau, \psi) \frac{L(1-s, \wt{\pi} \times \wt{\tau})}{L(s, \pi \times \tau)}\,.$$
Then the functional equation in Part (ii) of Theorem \ref{rsconv} can be written as 
\begin{equation}\label{fe}
\Psi(1-s,\rho(\omega_{n,t})\wt{W},\wt{W'};k)
= \omega_{\tau}(-1)^{n-1} \gamma(s, \pi \times \tau, \psi) \Psi(s,W,W';j)\,.
\end{equation}

At the end of this section, we introduce the following important lemma. 

\begin{lem}\label{lem4}
Let $\pi_1$ and $\pi_2$ be two irreducible generic representations of $G_n$. Let $t \leq n-2$ and $j$ with $0 \leq j \leq t$. Suppose that $W^1$ and $W^2$ are elements in the Whittaker models of $\pi_1$ and $\pi_2$ respectively. Suppose further that for all irreducible generic representations $\tau$ of $G_{n-t-1}$ we have
$$\Psi(s, W^1, W'; j)=\Psi(s, W^2, W'; j)$$
for all $W' \in \CW(\tau, \ol{\psi})$ and for $\mathrm{Re}(s) \gg 0$. Then
$$\int W^1\begin{pmatrix}
I_{n-t-1} & 0 & 0 \\
X & I_j & 0\\
0 & 0 & I_{t+1-j}
\end{pmatrix} dX = \int W^2\begin{pmatrix}
I_{n-t-1} & 0 & 0 \\
X & I_j & 0\\
0 & 0 & I_{t+1-j}
\end{pmatrix} dX\,,$$
where the integrals are over $X \in M_{j \times (n-t-1)}(F)$. 
\end{lem}

\begin{proof}
For $j=0$, the assumption is that 
\begin{align}\label{sec2equ1}
\begin{split}
\ & \int_{U_{n-t-1} \bs G_{n-t-1}} W^1\begin{pmatrix}
g & 0 \\
0 & I_{t+1}
\end{pmatrix} W'(g) \lvert \det g \rvert^{s+\frac{t+1}{2}} dg\\
= \ & \int_{U_{n-t-1} \bs G_{n-t-1}} W^2\begin{pmatrix}
g & 0 \\
0 & I_{t+1}
\end{pmatrix} W'(g) \lvert \det g \rvert^{s+\frac{t+1}{2}} dg\,,
\end{split}
\end{align}
for all $W'$. The conclusion is that $W^1(I_n)=W^2(I_n)$. 
Indeed, recall that given $C>0$ the relations
\[ \lvert \det g\rvert=C\,,\, W^i  \left(\begin{array}{cc} g & 0 \\ 0 & I_{t+1} \end{array}\right) \neq 0\]
imply that $g$ is in a set compact modulo $U_{n-t-1}$. Both sides of the identity \eqref{sec2equ1} converge for $\mathrm{Re}(s) \gg 0$. Thus they can be interpreted as formal Laurent series in $q^{-s}$. We conclude that for any $C>0$
\[ \int_{\lvert \det g \rvert=C }W^1 \left(\begin{array}{cc} g & 0 \\ 0 & I_{t+1} \end{array}\right) W'(g) dg = \int _{\lvert \det g \rvert=C}W^2 \left(\begin{array}{cc} g & 0 \\ 0 & I_{t+1} \end{array}\right) W'(g)  d g\,. \]
One then applies the spectral theory of the space $L^2 (U_{n-t-1}\backslash G_{n-t-1}^0)$ where $ G_{n-t-1}^0=\{ g \in G_{n-t-1}:\lvert \det g \rvert=1\}$. For more details, see \cite[Section 3]{H93} and \cite[Section 2]{Ch06}.

For $0 < j \leq t$, one observes that there is a compact subset $\Omega$ of $M_{j \times (n-t-1)}(F)$ such that for all $g \in G_{n-t-1}$ and $i=1,2$, 
$$W^i \begin{pmatrix}
g & 0 & 0 \\
X & I_j & 0\\
0 & 0 & I_{t+1-j}
\end{pmatrix}
\neq 0$$
implies that $X \in \Omega$. 
Thus, for $i=1,2$, there is an element $W^i_0 \in \CW(\pi_i, \psi)$ such that for all $g \in G_{n-t-1}$
$$\int_{M_{j \times (n-t-1)}(F)} W^i \begin{pmatrix}
g & 0 & 0 \\
X & I_j & 0\\
0 & 0 & I_{t+1-j}
\end{pmatrix} dX= W^i_0 \begin{pmatrix}
g & 0 & 0 \\
0 & I_j & 0\\
0 & 0 & I_{t+1-j}
\end{pmatrix}\,.$$
We are therefore reduced to the case $j=0$. 
\end{proof}

\section{Proof of Theorem \ref{main}}

In this section, we prove Theorem \ref{main}. 
Let $\pi_1$ and $\pi_2$ be irreducible generic representations of $G_n$ with the same central character $\omega$. 
We recall from Section 2.1 that $P$ is the maximal parabolic subgroup of $G_n$ with Levi subgroup $G_{n-1} \times G_1$, $Z$ is the center of $G_n$, $V_0={ \rm ind}_{ZU}^P \omega\psi_U$, and we have
\begin{equation}\label{sec3equ1}
W^1_v(p)=W^2_v(p)\,,\, \forall p \in P\,,\,\forall v \in V_0\,,
\end{equation}
\begin{equation}\label{sec3equ2}
W^i_v(gp)=W^i_{\rho(p)v}(g)\,,\, \forall g \in G_n\,,\, \forall p \in P\,,\, \forall v \in V_0\,,\, i=1,2\,.
\end{equation}

We recall the decomposition of $G_n$ into double cosets of $U$ and $P$ as in \cite{Ch06}:
\[G_n = \dot\bigcup_{i=0}^{n-1}U \alpha^i P\,,\]
where 
$$\alpha=\begin{pmatrix}
0 & I_{n-1}\\
1 & 0
\end{pmatrix}\,.$$
Note that $\alpha^i=\begin{pmatrix}
0 & I_{n-i}\\
I_i & 0
\end{pmatrix}$, in particular, $\alpha^0=\alpha^n=I_n$. 

\begin{defn}\label{defn}
For each double coset $U \alpha^i P$, $0 \leq i \leq n-1$, we call $i$ the {\it height} of the double coset. We say that $\pi_1$ and $\pi_2$ {\it agree} at height $i$ if 
$$
W^1_v(g)=W^2_v(g)\,,\, \forall g \in U\alpha^i P\,,\, \forall v \in V_0\,.$$
\end{defn}

By \eqref{sec3equ1}, $\pi_1$ and $\pi_2$ agree at height 0.
The following lemma gives a characterization of $\pi_1$ and $\pi_2$ agreeing at height $i$.

\begin{lem}[\cite{Ch06}, Lemma 3.1]\label{lem5}
$\pi_1$ and $\pi_2$ agree at height $i$
if and only if
 $$W^1_v(\alpha^i)=W^2_v(\alpha^i)\,,\, \forall v \in V_0\,.$$
\end{lem} 

The following proposition is one of the main ingredients for this paper. 

\begin{lem}[\cite{Ch06}, Proposition 3.1]\label{lem7}
Let $t$ with $1 \leq t \leq n-1$. If $\pi_1$ and $\pi_2$ satisfy hypothesis $\CH_t$, then
they agree at height $t$. 
\end{lem}

To proceed, we give a characterization of the matrices in the double coset $P \alpha^s U$, $0 \leq s \leq n-1$.

\begin{lem}\label{lem1}
Suppose $0 \leq s \leq n-1$, Then $g \in P \alpha^s U$ if and only if the last row of $g$ has the form 
$$(0, \ldots, 0, a_{s}, a_{s+1}, \ldots, a_n)\,,\, a_{s} \neq 0\,.$$
\end{lem}

\begin{proof}
Recall that
$$\alpha^{s} = \begin{pmatrix}
0 & I_{n-s}\\
I_{s} & 0
\end{pmatrix}\,.$$
It is clear that the last row of any matrix in $P \alpha^{s}$  has the form 
$$(0, \ldots, 0, a_{s}, 0, \ldots, 0)\,,$$
where $a_s \neq 0$ occurs in the $s$-th column of the matrix.
After multiplying by matrices in $U$ from the right, one can see that  
last row of any matrix in $P \alpha^{s} U$  has the form $(0, \ldots, 0, a_{s}, a_{s+1}, \ldots, a_n)$, with $a_{s} \neq 0$. 
\end{proof}

In fact, this lemma gives at once the decomposition in the disjoint double cosets
\[G_n = \dot\bigcup_{i=0}^{n-1}P \alpha^i U = \dot\bigcup_{i=0}^{n-1}U \alpha^i P\,.\]

The next lemma is a generalization of \cite[Lemma 3.2]{Ch06}. 

\begin{lem}\label{lem2}
Let $t$ with $[\frac{n}{2}] \leq t \leq n-2$. Suppose that for any $s$ with $0 \leq s \leq t$ the representations $\pi_1$ and $\pi_2$ agree at height $s$. Then the following equality holds
for all $X \in M_{(n-t-1) \times (2t+2-n)}(F)$, all $g \in G_{n-t-1}$, and all $v \in V_0$:
\begin{equation*}
W_v^1 \begin{pmatrix}
I_{n-t-1} & 0 & 0\\
0 & I_{2t+2-n} & 0\\
0 & X & g
\end{pmatrix}=W_v^2 \begin{pmatrix}
I_{n-t-1} & 0 & 0\\
0 & I_{2t+2-n} & 0\\
0 & X & g
\end{pmatrix}\,.
\end{equation*}
\end{lem}

\begin{proof}
First note that the hypothesis $[\frac{n}{2}] \leq t \leq n-2$ implies that $n-t-1 \geq 1$ and $2t+2-n \geq 1$. 

Let $A = \begin{pmatrix}
I_{n-t-1} & 0 & 0\\
0 & I_{2t+2-n} & 0\\
0 & X & g
\end{pmatrix}$, where $X \in M_{(n-t-1) \times (2t+2-n)}(F)$, and $g \in G_{n-t-1}$. Then $A^{-1} = \begin{pmatrix}
I_{n-t-1} & 0 & 0\\
0 & I_{2t+2-n} & 0\\
0 & -g^{-1}X & g^{-1}
\end{pmatrix}$. By Lemma \ref{lem1}, $A^{-1} \in P \alpha^i U$, where $i \geq n-t$, hence, $A \in U \alpha^{n-i} P$ with $n-i \leq t$. Since $\pi_1$ and $\pi_2$ agree at heights $0, 1, 2, \ldots, t$, $W_v^1$ and $W_v^2$ agree on $A$, for any $v \in V_0$.  

This completes the proof of the lemma.
\end{proof}

The following proposition allows us to prove Theorem \ref{main}    
 inductively. 

\begin{prop}\label{prop3}
Assume that $\pi_1$ and $\pi_2$ satisfy hypothesis $\CH_{\leq [\frac{n}{2}]}$. Let $t$ with $[\frac{n}{2}] \leq t \leq n-2$. Suppose that for any $s$ with $0 \leq s \leq t$, the representations $\pi_1$ and $\pi_2$ agree at height $s$. Then they agree at height $t+1$.  
\end{prop}

Before proving the proposition, we apply it to the proof of our main result as follows.

\textbf{Proof of Theorem \ref{main}}. 
Assume that $\pi_1$ and $\pi_2$ satisfy hypothesis $\CH_{\leq [\frac{n}{2}]}$. By Lemma \ref{lem7}, $\pi_1$ and $\pi_2$ agree at heights $1, 2, \ldots, [\frac{n}{2}]$. Note that by \eqref{sec3equ1}, $\pi_1$ and $\pi_2$ already agree at height 0.
Applying Proposition \ref{prop3} repeatedly for $t$ from $[\frac{n}{2}]$ to $n-2$, we obtain that $\pi_1$ and $\pi_2$ also agree at heights $[\frac{n}{2}]+1, \ldots, n-1$. 
Hence, $\pi_1$ and $\pi_2$ agree at all the heights $0, 1, \ldots, n-1$, that is, $W^1_v(g)=W^2_v(g)$, for all $g \in G_n$ and for all $v \in V_0$. Therefore, $\pi_1 \cong \pi_2$. 
This concludes the proof of Theorem \ref{main}. \qed

Therefore, we only need to prove Proposition \ref{prop3}, which will be done in Section 4.

\section{Proof of Proposition \ref{prop3}}

In this section, we prove Proposition \ref{prop3}.

\textbf{Proof of Proposition \ref{prop3}}.
By Lemma \ref{lem2}, 
\begin{equation*}
W_v^1 \begin{pmatrix}
I_{n-t-1} & 0 & 0\\
0 & I_{2t+2-n} & 0\\
0 & X & g
\end{pmatrix}=W_v^2 \begin{pmatrix}
I_{n-t-1} & 0 & 0\\
0 & I_{2t+2-n} & 0\\
0 & X & g
\end{pmatrix}
\end{equation*}
holds for all $X \in M_{(n-t-1) \times (2t+2-n)}(F)$, all $g \in G_{n-t-1}$, and all $v \in V_0$. Fix any pair $(X,g)$. Then, 
\begin{equation*}
W_v^1 \left(\omega_n \omega_n \begin{pmatrix}
I_{n-t-1} & 0 & 0\\
0 & I_{2t+2-n} & 0\\
0 & X & g
\end{pmatrix}\right)=W_v^2 \left(\omega_n \omega_n \begin{pmatrix}
I_{n-t-1} & 0 & 0\\
0 & I_{2t+2-n} & 0\\
0 & X & g
\end{pmatrix}\right)\,,
\end{equation*}
that is, 
\begin{equation*}
W_v^1 \left(\omega_n \begin{pmatrix}
g_1 & X_1 & 0\\
0 & I_{2t+2-n} & 0\\
0 & 0 & I_{n-t-1}
\end{pmatrix}\omega_n\right)=W_v^2 \left(\omega_n \begin{pmatrix}
g_1 & X_1 & 0\\
0 & I_{2t+2-n} & 0\\
0 & 0 & I_{n-t-1}
\end{pmatrix}\omega_n\right)\,,
\end{equation*}
where $g_1 = \omega_{n-t-1} g \omega_{n-t-1}$, $X_1=\omega_{n-t-1} X \omega_{2t+2-n}$. 

Note that 
$$\omega_n = 
\begin{pmatrix}
\omega_{n-t-1} & 0\\
0 & I_{t+1}
\end{pmatrix} \omega_{n,n-t-1} \alpha^{t+1}\,.$$
Recall that 
$$\omega_{n,n-t-1}=\begin{pmatrix}
I_{n-t-1} & 0\\
0 & \omega_{t+1}
\end{pmatrix}\,.$$
Hence,
\begin{align*}
\ & W_v^1 \left(\omega_n \begin{pmatrix}
g_2 & X_1 & 0\\
0 & I_{2t+2-n} & 0\\
0 & 0 & I_{n-t-1}
\end{pmatrix}\omega_{n,n-t-1} \alpha^{t+1}\right)\\
= \ &  W_v^2 \left(\omega_n \begin{pmatrix}
g_2 & X_1 & 0\\
0 & I_{2t+2-n} & 0\\
0 & 0 & I_{n-t-1}
\end{pmatrix}\omega_{n,n-t-1} \alpha^{t+1}\right)\,,
\end{align*}
where $g_2 = \omega_{n-t-1} g$, $X_1=\omega_{n-t-1} X \omega_{2t+2-n}$. 

Let $X_v^i= \rho(\alpha^{t+1}) W^i_{v}$. Then
\begin{align*}
& X_v^1 \left(\omega_n \begin{pmatrix}
g_2 & X_1 & 0\\
0 & I_{2t+2-n} & 0\\
0 & 0 & I_{n-t-1}
\end{pmatrix}\omega_{n,n-t-1}\right)\\
= \ & X_v^2 \left(\omega_n \begin{pmatrix}
g_2 & X_1 & 0\\
0 & I_{2t+2-n} & 0\\
0 & 0 & I_{n-t-1}
\end{pmatrix}\omega_{n,n-t-1}\right)\,.
\end{align*}
Recall that $\wt{X^i_v}(g)=X^i_v(\omega_n {}^t g^{-1})$. 
Then, 
\begin{align*}
& \wt{X_v^1} \left(\begin{pmatrix}
g_3 & 0 & 0\\
X_2 & I_{2t+2-n} & 0\\
0 & 0 & I_{n-t-1}
\end{pmatrix}\omega_{n,n-t-1}\right)\\
= \ & \wt{X_v^2} \left(\begin{pmatrix}
g_3 & 0 & 0\\
X_2 & I_{2t+2-n} & 0\\
0 & 0 & I_{n-t-1}
\end{pmatrix}\omega_{n,n-t-1}\right)\,,
\end{align*}
where $g_3 = \omega_{n-t-1} {}^t g^{-1}$, $X_2=-\omega_{2t+2-n} {}^t X \omega_{n-t-1} g_1$. 

Therefore, 
\begin{align*}
& \wt{X_v^1} \left(\begin{pmatrix}
g & 0 & 0\\
X & I_{2t+2-n} & 0\\
0 & 0 & I_{n-t-1}
\end{pmatrix}\omega_{n,n-t-1}\right)\\
= \ & \wt{X_v^2} \left(\begin{pmatrix}
g & 0 & 0\\
X & I_{2t+2-n} & 0\\
0 & 0 & I_{n-t-1}
\end{pmatrix}\omega_{n,n-t-1}\right)\,,
\end{align*}
for all $X \in M_{(2t+2-n) \times (n-t-1)}(F)$, all $g \in G_{n-t-1}$, and all $v \in V_0$. Then, by the definition of the zeta integral $\Psi$ in \eqref{zeta}, 
we have the following equality:
\begin{align*}
& \Psi(1-s,\rho(\omega_{n,n-t-1})(\wt{X^1_v}), \wt{W_{\tau}}; 2t+2-n) \\
= \ & \Psi(1-s,\rho(\omega_{n,n-t-1})(\wt{X^2_v}), \wt{W_{\tau}}; 2t+2-n)\,,
\end{align*}
for all irreducible generic representations $\tau$ of $G_{n-t-1}$, all Whittaker functions $W_{\tau} \in \CW(\tau, \ol{\psi})$, and all $v \in V_0$. Note that the above equality first holds for $\mathrm{Re}(s) \ll 0$ and is then an identity of rational functions of $q^{-s}$ for all $\tau$, all $W_{\tau}$, and all $v \in V_0$.

Since $\pi_1$ and $\pi_2$ satisfy hypothesis $\CH_{\leq [\frac{n}{2}]}$, and $n-t-1 \leq [\frac{n}{2}]$, 
by functional equation in \eqref{fe}, we have that 
$$\Psi(s, X^1_v, W_{\tau}; n-t-2) = \Psi(s,X^2_v, W_{\tau}; n-t-2)\,,$$
for all irreducible generic representations $\tau$ of $G_{n-t-1}$, all Whittaker functions $W_{\tau} \in \CW(\tau, \ol{\psi})$, and all $v \in V_0$.
Hence, by Lemma \ref{lem4}, 
\begin{align*}
& \int_{M_{(n-t-2) \times (n-t-1)}(F)} X^1_v \begin{pmatrix}
I_{n-t-1} & 0 & 0\\
X & I_{n-t-2} & 0\\
0 & 0 & I_{2t+3-n}
\end{pmatrix} dX \\
= \ & \int_{M_{(n-t-2) \times (n-t-1)}(F)} X^2_v \begin{pmatrix}
I_{n-t-1} & 0 & 0\\
X & I_{n-t-2} & 0\\
0 & 0 & I_{2t+3-n}
\end{pmatrix} dX\,,
\end{align*}
for all $v \in V_0$. 
We claim (Lemma \ref{lem3} below) that this identity implies in fact
$$X^1_v(I_n)=X^2_v(I_n)\,,\, \forall v \in V_0\,.$$
Taking this for granted at the moment we finish the proof. Indeed, we have then
$$W^1_v(\alpha^{t+1})=W^2_v(\alpha^{t+1})\,,\, \forall v \in V_0\,.$$
Therefore by Lemma \ref{lem5}, $\pi_1$ and $\pi_2$ agree at height $t+1$. 
This concludes the proof of Proposition \ref{prop3}. \qed

In the rest of this paper we establish our claim, that is, we prove the following lemma.
We remark that in the case $t=n-2$, considered in \cite{Ch06}, there is no need of the following lemma, since $n-t-2=0$ when $t=n-2$.

\begin{lem}\label{lem3}
Recall that $X_v^i= \rho(\alpha^{t+1}) W^i_{v}$, $i=1,2$. If 
\begin{align}\label{lem3equ1}
\begin{split}
& \int_{M_{(n-t-2) \times (n-t-1)}(F)} X^1_v \begin{pmatrix}
I_{n-t-1} & 0 & 0\\
X & I_{n-t-2} & 0\\
0 & 0 & I_{2t+3-n}
\end{pmatrix} dX \\
= \ & \int_{M_{(n-t-2) \times (n-t-1)}(F)} X^2_v \begin{pmatrix}
I_{n-t-1} & 0 & 0\\
X & I_{n-t-2} & 0\\
0 & 0 & I_{2t+3-n}
\end{pmatrix} dX\,,
\end{split}
\end{align} 
for all $v \in V_0$, 
then $X^1_v(I_n)=X^2_v(I_n)$, for all $v \in V_0$.
\end{lem}

\begin{proof}
Since $X_v^i= \rho(\alpha^{t+1}) W^i_{v}$, by \eqref{sec3equ2}, 
equality \eqref{lem3equ1} implies that 
\begin{align}\label{lem3equ2}
\begin{split}
& \int_{M_{(n-t-2) \times (n-t-1)}(F)} X^1_v \left(u \begin{pmatrix}
I_{n-t-1} & 0 & 0\\
X & I_{n-t-2} & 0\\
0 & 0 & I_{2t+3-n}
\end{pmatrix} p\right) dX \\
= \ & \int_{M_{(n-t-2) \times (n-t-1)}(F)} X^2_v \left(u \begin{pmatrix}
I_{n-t-1} & 0 & 0\\
X & I_{n-t-2} & 0\\
0 & 0 & I_{2t+3-n}
\end{pmatrix} p\right) dX\,,
\end{split}
\end{align} 
for all $u \in U$, all $p \in \alpha^{t+1} P (\alpha^{t+1})^{-1}$, and all $v \in V_0$. 
Recall that 
$$\alpha^{t+1} = \begin{pmatrix}
0 & I_{n-t-1}\\
I_{t+1} & 0
\end{pmatrix}\,.$$
Hence the $(n-t-1)$-th row of any $p$ in $\alpha^{t+1} P (\alpha^{t+1})^{-1}$ has the form $(0, \dots, 0, a, 0, \ldots, 0)$ with $a \neq 0$ in the $(n-t-1)$-th column. 
Conversely, this condition characterizes the elements of 
$\alpha^{t+1} P (\alpha^{t+1})^{-1}$.
We will use the relation \eqref{lem3equ2} only for  $p\in U\cap \alpha^{t+1} P (\alpha^{t+1})^{-1}$. 

We denote by $\xi_{i,j}$ the matrix whose only non-zero entry 
is $1$ in the $i$-th row and $j$-th column. Thus
\[ \xi_{i,j} \xi_{j',k}= \delta _{j,j'} \xi_{i,k} \,.\]
Given a root $\alpha$ (positive or negative) we denote by $X_\alpha$ the corresponding root subgroup. Thus if $\alpha=e_i-e_j$, for any $a\in F$, the element $I_n+ a \xi_{i,j}$ is in $X_\alpha$.

Set 
\[\frak{X}= \left\{\left( \begin{array}{c c c} I_{n-t-1} & 0 & 0 \\ X & I_{n-t-2} \\ 0&0& I_{2 t + 3 -n} \end{array}\right)\,,\,
X \in M_{(n-t-2) \times (n-t-1)}(F) \right\}\,.\]
The group $\frak{X}$ is abelian and is the direct product of the groups $X_{e_a-e_b}$ with
\[ n-t\leq a \leq 2(n-t)-3\,,\, 1 \leq b\leq n -t-1 \,.\]
For such a pair $(a,b)$ we have either
\[ b\leq a-(n-t)+1 \]
or
\[ a \le b + n-t-2 \,.\]
We then define subgroups of $\frak{X}$ as follows.
For $n -t \leq a \leq 2(n-t)-3$, we define the following subgroup of $\frak{X}$:
\[ X_a= \prod_{1 \leq b \leq a -(n-t) +1} X_{e_a-e_b}\,.\]
We also define a subgroup of $U$  as follows. 
\[  Y_a = \prod_{1 \leq b \leq a -(n-t) +1} X_{e_b -e _{a+1}} \,.\]
We can identify $Y_a$ with the dual of $X_a$ as follows: if for $X\in X_a$, $Y\in Y_a$, write
\begin{align*}
X =\ & I_n+ \sum _{1 \leq b \leq a -(n-t) +1} \xi_{a,b} x_b \,,\\
Y =\ & I_n + \sum_{ 1 \leq b \leq a -(n-t) +1} \xi_{b,a+1} y_b\,,
\end{align*}
then set
\[ \langle X,Y\rangle = \sum_{ 1 \leq b \leq a -(n-t) +1} x_b y_b \,.\]

We remark that $Y_a$ is contained in the subgroup $U \cap \alpha^{t+1} P
\alpha ^{-(t+1)}$. Indeed, $b$ cannot take the value $n-t-1$, otherwise, we would have $n-t-1\leq a-(n-t)+1$ or $2(n-t)-2 \leq a$, which contradicts the assumption $a\leq 2(n-t)-3$. 

For $2\leq b \leq n-t-1$, we define
\[ Z_b = \prod_{n-t\leq a \leq b+ n-t-2} X_{e_a-e_b} \,.\]
We also define a subgroup of $U$  as follows:
\[ T_b = \prod_{n-t\leq a \leq b+ n-t-2} X_{e_{b-1}-e_a} \,.\]
Again we can identify $T_b$ with the dual of $Z_b$ as follows: if for $Z\in Z_b$, $T\in T_b$, write
\begin{align*}
Z =\ & I_n+  \sum_{ n-t\leq a \leq b+ n-t-2} \xi _{a,b} z_a \,,\\
T= \ & I_n + \sum_{ n-t\leq a \leq b+ n-t-2} \xi _{b-1,a} t_a\,,
\end{align*}
then set
\[ \langle Z , T\rangle = \sum_{ n-t\leq a \leq b+ n-t-2} z_a y_a \,.\]

Since $b-1\leq n-t-2$, the $(n-t-1)$-th row of a matrix in $T_b$ has all its elements $0$ except the diagonal element equal to $1$. Thus $T_b$ is contained in $U \cap \alpha^{t+1}P \alpha^{-(t+1)}$.

The group $\frak{X}$ is the product
$$\prod_{n-t \leq a \leq 2(n-t)-3} X_{a} \prod_{2 \leq b \leq n-t-1} Z_b\,.$$
The identity \eqref{lem3equ1} can be written as follows: for all $v \in V_0$, 
\[\int_{\frak{X}} X_v^1(X) d X = \int_{\frak{X}} X_v^2 (X ) d X \,.\] 
Note that the two functions $X^i_v$ on $\frak{X}$ are smooth and compactly supported. We should keep in mind that
\[ X_{\rho(p)v}^i (X) = X_v^i( X (\alpha^{t+1} p {\alpha^{-(t+1)}} )) \,,\, \forall p \in P\,,\, \forall v \in V_0\,.\]

\textbf{First step.} We show that we have, for all $v \in V_0$, the identity
\[ \int X_v^1(X) d X = \int X_v^2 (X ) d X \, ,\]
where both integrals are over the product
$$\prod_{n-t \leq a \leq 2(n-t)-4} X_{a} \prod_{2 \leq b \leq n-t-1} Z_b\,.$$

By \eqref{lem3equ2}, for all $Y \in Y_{2(n-t)-3}=\prod _{1\leq b \leq n-t-2} X_{e_b -e _{2(n-t)-2}}$ and all $v \in V_0$, we have
 \[\int X_v^1(XY ) d X = \int X_v^2 (X Y ) d X \,,\] 
where both integrals are over the product
$$\prod_{n-t \leq a \leq 2(n-t)-3} X_{a} \prod_{2 \leq b \leq n-t-1} Z_b\,.$$  
We write
\[X = \left(\begin{array}{ccc} I_{n-t-1} & 0 & 0 \\ A & I_{ n-t-2}&0 \\ 0 &0 &I_{ 2t + 3 -n} \end{array}\right) \,,\]
\[ Y = \left(\begin{array}{ccc} I_{n-t-1} & 0 & B \\ 0 & I_{ n-t-2}&0 \\ 0 &0 &I_{ 2t + 3 -n} \end{array}\right) \,.\]
Then
\[ XY = \left(\begin{array}{ccc} I_{n-t-1} & 0 & B \\ 0 & I_{ n-t-2}& AB \\ 0 &0 &I_{ 2t + 3 -n} \end{array}\right)  \left(\begin{array}{ccc} I_{n-t-1} & 0 & 0 \\ A & I_{ n-t-2}&0 \\ 0 &0 &I_{ 2t + 3 -n} \end{array}\right) \,.\]
We must evaluate
\[ \psi_U \left(\begin{array}{ccc} I_{n-t-1} & 0 & B \\ 0 & I_{ n-t-2}& AB \\ 0 &0 &I_{ 2t + 3 -n} \end{array}\right) \,.\]
We write
\[\left(\begin{array}{ccc} 0_{n-t-1} & 0 & 0 \\ A & 0_{ n-t-2}&0 \\ 0 &0 &0_{ 2t + 3 -n} \end{array}\right)
= \sum _{n-t\leq a \leq 2(n-t)-3\,,\, 1 \leq b\leq n-t-1} \xi_{a,b} x _{a,b} \,.\]
By abuse of notations, we write this in the form
\[A = \sum _{n-t\leq a \leq 2(n-t)-3\,,\, 1 \leq b\leq n-t-1} \xi_{a,b} x _{a,b} \,.\]
Similarly,  
\[ B = \sum _{1\leq j \leq n-t-2} \xi_{j, {2(n-t)-2}}y_j \,.\]
Hence
\[ AB = \sum _{n-t\leq a \leq 2(n-t)-3} \xi_{a, 2(n-t)-2} \left( \sum_{1\leq j\leq n-t-2} x_{a,j} y_j\right) \,,\]
and
\begin{align*}
& \psi_U \left(\begin{array}{ccc} I_{n-t-1} & 0 & B \\ 0 & I_{ n-t-2}& AB \\ 0 &0 &I_{ 2t + 3 -n} \end{array}\right)  \\
= \ & \psi\left(\sum_{1\leq j\leq n-t-2} x_{2(n-t)-3,j} y_j \right)\\
= \ & \psi(\langle X^{2(n-t)-3}, Y\rangle)\,, 
\end{align*}
where $ X^{2(n-t)-3}$ is the projection of $X$ on the subgroup $X_{2(n-t)-3}$. 
Hence we have, for all $Y \in Y_{2(n-t)-3}$ and all $v \in V_0$,
\[ \int X_v^1 (X)  \psi(\langle X^{2(n-t)-3}, Y\rangle) d X =\int X_v^2 (X) \psi(\langle X^{2(n-t)-3}, Y\rangle)  d X \,,\]
where both integrals are over the product
$$\prod_{n-t \leq a \leq 2(n-t)-3} X_{a} \prod_{2 \leq b \leq n-t-1} Z_b\,.$$
Applying Fourier inversion formula on the group $X_{2(n-t)-3}$, we obtain our assertion.

\textbf{Second step.}
 Assume that for $k$ with $n-t\leq k\leq 2(n-t)-4$ and for all $v \in V_0$, we have established the identity 
 \[ \int X^1_v(X) d X = \int X^2_v(X) d X\,,\]
where both integrals are over the product
\[ \prod_{n-t \leq a \leq k\,,\, 2 \leq b \leq n-t-1} X_a Z_ b\,.\]
We show that for all $v \in V_0$, we have the identity
 \[ \int X^1_v(X) d X = \int X^2_v(X) d X\,,\]
where both integrals are over the product
\[ \prod_{n-t \leq a \leq k-1\,,\, 2 \leq b \leq n-t-1} X_a Z_ b\,.\]
 
By \eqref{lem3equ2}, for all $v \in V_0$ and all $Y\in Y_k$, we have
\[ \int X^1_v(X Y) d X = \int X^2_v(X Y) d X\,,\]
where both integrals are over the product
\[ \prod_{n-t \leq a \leq k\,,\, 2 \leq b \leq n-t-1} X_a Z_ b\,.\]
Recall that
\begin{align*}
X_k = \ & \prod _{1\leq b \leq k-(n-t) + 1 } X_{e_k-e_b}\,,\\
Y_k = \ & \prod _{1 \le b \leq k -(n-t) +1} X _{e_b -e_{k+1}}\,.
\end{align*}
Hence $ n-t+1 \leq k+1 \leq 2(n-t)-3$ and $b\leq n-t-3$. Thus we may write $Y$ as the matrix
\[Y= \left(\begin{array}{ccc} I_{n-t-1} & B & 0 \\ 0 & I_{ n-t-2}&0 \\ 0 &0 &I_{ 2t + 3 -n} \end{array}\right) \,.\]
We still write
\[X = \left(\begin{array}{ccc} I_{n-t-1} & 0 & 0 \\ A & I_{ n-t-2}&0 \\ 0 &0 &I_{ 2t + 3 -n} \end{array}\right) \,.\]
Then
\[ XY = \left(\begin{array}{ccc} I_{n-t-1} & B & 0 \\ A & I_{ n-t-2}+ AB&0 \\ 0 &0 &I_{ 2t + 3 -n} \end{array}\right) \,.\]

To continue we must check that the matrix $ I_{ n-t-2}+ AB$ is invertible.  Now again by abuse of notations as in the First step, write
\[ A = \sum_{1\leq b\leq a -(n-t) +1\,,\, n-t \leq a \leq k} \xi_{ a,b} x_{a,b}+ \sum_{2 \leq b \leq n-t-1\,,\, n-t \leq a \leq b+ n-t-2} \xi_{a,b} z_{a,b}\,,\]
and
\[ B= \sum _{ 1 \leq j \leq k-(n-t)+1}\xi_{j, k+1} y_j \,.\]
In the product $AB$, the contribution of the first sum in $A$ is 
$$\sum _{n-t \leq a \leq k} \xi_{a, k+1} \left( \sum _{1 \leq j\leq k -(n-t)+1}  x_{a,j} y_j \right)\,.$$
The contribution to $AB$ of the second sum in $A$ is itself a sum of terms of the form
$$\xi_{a,k+1} z_{a,j} y_j\,,$$
with 
$$n-t \leq a \leq j+n-t-2, 2 \leq j\leq k-(n-t)+1\,,$$
inequalities which imply that $a \leq k-1$.
We conclude that
\[ AB = \sum _{ n-t\leq a \leq k} \xi_{a,k+1} m_a\,, \]
where
\[ m_k = \sum _{1 \leq j\leq k -(n-t)+1}  x_{k,j} y_j= \langle X^k,Y\rangle\,,\]
and $X^k$ is the projection of $X$ on the group $X_k$. 
Thus $I_{ n-t-2}+ AB$ is invertible and in fact, since $AB$ has only one non-zero column, 
\[ (I_{n-t-2}+ AB)^{-1} = I_{n-t-2}- AB\,.\]

We introduce the matrix
\[ \widetilde{A} = (I_{n-t-2} +A B)^{-1}A = (I_{n-t-2} -AB)A\,.\]
We compute $ABA$. The first sum in $A$ does not contribute to the product of $AB$ by $A$. The second sum contributes 
$$\xi_{a,b} z_{k+1,b}m_a\,,$$
where the sum is for 
$$n-t \leq a \leq k, 2 \leq b \leq n-t-1, n-t \leq k+1 \leq b+n-t-2\,.$$
Recall that $n-t \leq k$. So, the range of $b$ is 
$$k-(n-t)+3 \leq b \leq n-t-1\,.$$
Thus
\[ ABA = \sum_{n-t \leq a \leq k\,,\, k-(n-t)+3 \leq b \leq n-t-1} \xi_{a,b} z_{k+1,b}m_a\,.\]
The pairs $(a,b)$ which appear satisfy the inequalities
\[2 \leq b \leq n-t-1 \,,\,  n-t \leq a \leq b + n-t -2\,.\]
We conclude that
\begin{align*}
& \widetilde{A} \\
= \ & A -A BA \\
= \ & \sum_{1\leq b\leq a -(n-t) +1\,,\, n-t \leq a \leq k} \xi_{ a,b} x_{a,b}+ \sum_{2 \leq b \leq n-t-1\,,\, n-t \leq a \leq b+ n-t-2} \xi_{a,b} z_{a,b}'\,.
\end{align*}
In words, $\widetilde{A}$ has the same shape as $A$ and the same $x_{a,b}$ coordinates. 
Hence the matrix
\[ \widetilde{X}= \left(\begin{array}{ccc} I_{n-t-1} & 0 & 0 \\ \widetilde{A} & I_{ n-t-2}&0 \\ 0 &0 &I_{ 2t + 3 -n} \end{array}\right) \]
is in the same group as the matrix $X$.
Also
\[ \widetilde{A} B = \sum _{ n-t\leq a \leq k} \xi_{a,k+1} \widetilde{m}_a\,,\]
and
\[ \widetilde{X}^k = X^k, \widetilde{m}_k = \langle\widetilde{X}^k,Y\rangle =\langle X^k,Y\rangle \,.\] 

On the other hand 
the matrix $BA$ is the sum of
\[ \xi_{b,b'} z_{k+1,b'} y_b \]
with 
\[ 1 \leq b \leq k-(n-t)+1\,,\, k-(n-t)+3 \leq b' \leq n-t-1\,.\]
The inequalities imply
\[ b + (n-t) \leq k+1 \leq b' + n-t -2 \]
or
\[ b \leq b'-2 \,.\]
Thus $BA$ is an upper triangular matrix  with $0$ entries in the diagonal and just above the diagonal. 
In particular, $I_{n-t-1}- BA$ is invertible. The same remarks apply to the matrix  $B\wt{A}$ and $I_{n-t-1}- B\wt{A}$.

Thus we can continue our computation
\begin{align*}
\ & XY \\
= \ & \left(\begin{array}{ccc} I_{n-t-1} -B \widetilde{A} & B & 0 \\ 0 & I_{ n-t-2}+ AB&0 \\ 0 &0 &I_{ 2t + 3 -n} \end{array}\right) \left(\begin{array}{ccc} I_{n-t-1} & 0 & 0 \\ \widetilde{A} & I_{ n-t-2}&0 \\ 0 &0 &I_{ 2t + 3 -n} \end{array}\right)\,,
\end{align*}
and we have
\begin{align*}
\ & \psi_U \left(\begin{array}{ccc} I_{n-t-1} -B \widetilde{A} & B & 0 \\ 0 & I_{ n-t-2}+ AB&0 \\ 0 &0 &I_{ 2t + 3 -n} \end{array}\right) \\
= \ & \psi_{U_{n-t-2}} ( I_{n-t-2} +AB) \\
= \ & \psi (\langle X^k,Y\rangle) \\
= \ & \psi (\langle \widetilde{X}^k,Y\rangle ) \,.
\end{align*}
Hence our identity reads
\[ \int X_v^1 (\widetilde{X} ) \psi (\langle \widetilde{X}^k,Y\rangle ) d X = \int X_v^2 (\widetilde{X} ) \psi (\langle \widetilde{X}^k,Y\rangle ) d X\,,\, \forall v \in V_0\,, \]
where both integrals are over the product
\[ \prod_{n-t \leq a \leq k\,,\, 2 \leq b \leq n-t-1} X_a Z_ b\,.\]

We want to use $\widetilde{X}$ as the variable of integration. Because $AB$ and $BA$ are nilpotent
we have
\[ d \widetilde{X} =\lvert p(X)\rvert d X \]
and
\[ d  X = \vert \widetilde{p} (\widetilde{X})\rvert d \widetilde {X} \]
where $p$ and $\widetilde{p}$ are polynomials in the entries of $X$ and $\widetilde{X}$ respectively.
Then
\[ \vert \widetilde{p}(\widetilde{X}) p(X)\rvert =1 \,.\]
Since $\widetilde{X}$ is a polynomial function of $X$ we see that $p$ is a constant $c>0$ and so $d \widetilde{X} =c d X$. In fact $c=1$ but we do not need this fact.
Hence our identity reads
\[ \int X_v^1 (\widetilde{X} ) \psi (\langle \widetilde{X}^k,Y\rangle ) d \widetilde{X} = \int X_v^2 (\widetilde{X} ) \psi (\langle \widetilde{X}^k,Y\rangle ) d \widetilde{X}\,,\, \forall v \in V_0\,,  \]
where both integrals are over the product
\[ \prod_{n-t \leq a \leq k\,,\, 2 \leq b \leq n-t-1} X_a Z_ b\,.\]

Applying Fourier inversion formula on the group $X_k$, we get,
 for all $v \in V_0$, the equality
\[ \int  X_v^1 (X) d X = \int  X_v^2 (X) d X\,, \]
where both integrals are over the product
\[ \prod_{n -t \leq a \leq k-1 \,,\,2 \leq b \leq n-t-1}  X_a Z_b \,.\]

\textbf{Third step.} Applying descending induction on $k$ we arrive at
\[ \int_{\prod_{2 \leq b \leq n-t-1}  Z_b} X_v^1 (Z) d Z = \int_{\prod_{2 \leq b \leq n-t-1}  Z_b} X_v^2 (Z) d Z\,,\, \forall v \in V_0\,. \]
We prove now that for $2\leq k \leq n-t-1$, if we have
\[ \int_{\prod_{ k \leq b \leq n-t-1}Z_b} X_v^1 (Z) d Z = \int_{\prod_{ k \leq b \leq n-t-1}Z_b} X_v^2 (Z) d Z\,,\, \forall v \in V_0\,,\]
then we have 
\[ \int_{\prod_{k+1 \leq b \leq n-t -1} Z_b} X_v^1 (Z) d Z = \int_{\prod_{k+1 \leq b \leq n-t -1} Z_b} X_v^2 (Z) d Z\,,\, \forall v \in V_0\,. \]
By ascending induction this will establish our contention.

By \eqref{lem3equ2}, we have for all $T\in T_k$ and all $v \in V_0$, 
\[ \int_{\prod_{k \leq b \leq n-t-1} Z_b} X_v^1 ( T Z T^{-1}) d Z = \int_{\prod_{k \leq b \leq n-t-1} Z_b} X_v^2 (TZT^{-1}) d Z\,.\]
We write
\[ Z = \left(\begin{array}{ccc} I_{n-t-1} & 0 & 0 \\ A & I_{n-t-2} & 0 \\0 &0 & I_{2t + 3 -n} \end{array}\right)\,,\]
and
\[ T= \left(\begin{array}{ccc} I_{n-t-1} & B & 0 \\ 0 & I_{n-t-2} & 0 \\0 &0 & I_{2t + 3 -n} \end{array}\right)\,. \]
Again by abuse of notations as in previous steps, write
\begin{align*}
A=\ & \sum_{k \leq b\leq n-t-1\,,\, n-t \leq a \leq b+ n-t-2} \xi_{a,b} z_{a,b} \,,\\
B=\ & \sum_{n-t\leq j\leq k+n-t-2} \xi_{k-1, j} t _j \,.
\end{align*}
Since $b\geq k$, the product $ \xi_{a,b} \xi_{k-1, j}$ is always $0$ and so $AB=0$. On the other hand
\[ BA= \sum_{ k \leq b\leq n-t-1} \xi_{k-1, b} \left( \sum _{ n-t\leq j\leq k+n-t-2}z_{j,b} t_j \right)\,.\] 
So $BA$ is upper triangular with zero diagonal. The only nonzero entry just above the diagonal is the coefficient of $\xi_{k-1,k}$ which is 
$$\sum _{ n-t\leq j\leq k+n-t-2}z_{j,k} t_j = \langle Z^k, T\rangle\,,$$
where $Z^k$ is the projection of $Z$ on the group $Z_k$. 
Thus
\[  \left(\begin{array}{ccc} I_{n-t-1}+BA  & 0 & 0 \\ 0 & I_{n-t-2} & 0 \\0 &0 & I_{2t + 3 -n} \end{array}\right)\in U\,, \]
and
\[ \psi_U\left(\begin{array}{ccc} I_{n-t-1}+BA  & 0 & 0 \\ 0 & I_{n-t-2} & 0 \\0 &0 & I_{2t + 3 -n} \end{array}\right) = \psi( \langle Z^k, T\rangle) \,.\]
Using the fact that $AB=0$, we find
\begin{align*}
\ & TZT^{-1}\\
 =\ & \left(\begin{array}{ccc} I_{n-t-1}+BA  & 0 & 0 \\ 0 & I_{n-t-2} & 0 \\0 &0 & I_{2t + 3 -n} \end{array}\right)\left(\begin{array}{ccc} I_{n-t-1} & 0 & 0 \\ A & I_{n-t-2} & 0 \\0 &0 & I_{2t + 3 -n} \end{array}\right)\,,
\end{align*}
and our identity reads, for all $T\in T_k$ and all $v \in V_0$, 
\[ \int_{\prod_{k\leq b \leq n-t-1} Z_b} X_v^1(Z)  \psi( \langle Z^k, T\rangle) d Z = \int_{\prod_{k\leq b \leq n-t-1} Z_b} X_v^2(Z)  \psi( \langle Z^k, T\rangle) d Z\,. \]
Applying Fourier inversion formula on the group $Z_k$, we conclude that 
\[ \int_{\prod_{k+1 \leq b \leq n-t-1} Z_b} X_v^1(Z) d Z = \int_{\prod_{k+1 \leq b \leq n-t-1} Z_b} X_v^2(Z) d Z\,,\, \forall v \in V_0\,. \]

This concludes the proof of the lemma. 
\end{proof}

\end{document}